\documentclass[12]{amsart}
\usepackage{amsmath,amssymb}
\usepackage{amsthm,color,enumerate,comment,centernot,enumitem,url,cite}
\usepackage{graphicx,relsize,bm}
\usepackage{mathtools}
\usepackage{array}
\usepackage{enumitem}
\setenumerate[0]{label=(\alph*)}

\makeatletter
\newcommand{\tpmod}[1]{{\@displayfalse\pmod{#1}}}
\makeatother

\newcommand{\ord}{\operatorname{ord}}

\newtheorem{thm}{Theorem}[section]
\newtheorem{lemma}[thm]{Lemma}

\newtheorem{prop}[thm]{Proposition}
\newtheorem{cor}[thm]{Corollary}

\theoremstyle{remark}

\theoremstyle{definition}
    \newtheorem{defn}[thm]{Definition}

\newtheorem{rem}[thm]{Remark}

\def\Z {{\mathbb Z}}

\def\NN {{\mathcal N}}

\def\Q {{\mathbb Q}}

\def\C {{\mathcal C}}
\def\S {{{\mathcal S}}}

\def\D {{\mathcal D}}

\def\F {{\mathbb F}}
\def\D {{\mathcal D}}
\def\Z {{\mathbb Z}}
\def\Q {{\mathbb Q}}
\def\C {{\mathbb C}}
\def\S {{{\mathcal S}}}

\def\CC {{\mathcal C}}

\makeatletter
\@namedef{subjclassname@2020}{%
  \textup{2020} Mathematics Subject Classification}
\makeatother

\def\red#1 {\textcolor{red}{#1 }}
\def\blue#1 {\textcolor{blue}{#1 }}

\numberwithin{equation}{section}

\def\Z {{\mathbb Z}}
\begin{document}

\title[Monogenicity of power-compositional Shanks polynomials]{On the monogenicity of power-compositional Shanks polynomials}

%\author{Joshua Harrington}
%\address{Department of Mathematics, Cedar Crest College, Allentown, Pennsylvania, USA}
%\email[Joshua Harrington]{Joshua.Harrington@cedarcrest.edu}

\author{Lenny Jones}
\address{Professor Emeritus, Department of Mathematics, Shippensburg University, Shippensburg, Pennsylvania 17257, USA}
\email[Lenny~Jones]{doctorlennyjones@gmail.com}

%\author{Daniel White}
%\address{Department of Mathematics, Bryn Mawr College, Bryn Mawr, Pennsylvania 19010-2899, USA}
%\email[Daniel~White]{dfwhite@brynmawr.edu}
\date{\today}

\begin{abstract}
 Let $f(x)\in {\mathbb Z}[x]$ be a monic polynomial of degree $N$ that is irreducible over ${\mathbb Q}$. We say $f(x)$ is \emph{monogenic} if  $\Theta=\{1,\theta,\theta^2,\ldots ,\theta^{N-1}\}$ is a basis for the ring of integers ${\mathbb Z}_K$ of $K={\mathbb Q}(\theta)$, where $f(\theta)=0$. If $\Theta$ is not a basis for ${\mathbb Z}_K$, we say that $f(x)$ is \emph{non-monogenic}.

Let $k\ge 1$ be an integer, and let $(U_n)$
be the sequence defined by
\[U_0=U_1=0,\quad U_2=1 \quad \mbox{and}\quad U_n=kU_{n-1}+(k+3)U_{n-2}+U_{n-3} \quad \mbox{for $n\ge 3$}.\]
It is well known that $(U_n)$ is periodic modulo any integer $m\ge 2$, and we let $\pi(m)$
denote the length of this period. We define a \emph{$k$-Shanks prime} to be a prime $p$ such that $\pi(p^2)=\pi(p)$.

  Let ${\mathcal S}_k(x)=x^{3}-kx^{2}-(k+3)x-1$  and 
  \begin{equation*}
  {\mathcal D}=\left\{\begin{array}{cl}
  (k/3)^2+k/3+1 & \mbox{if $k\equiv 0 \pmod{3}$,}\\
  k^2+3k+9 & \mbox{otherwise.}
\end{array} \right.
\end{equation*}
Suppose that $k\not \equiv 3 \pmod{9}$ and that ${\mathcal D}$ is squarefree. 
In this article, we prove that $p$ is a $k$-Shanks prime if and only if ${\mathcal S}_k(x^p)$ is non-monogenic, for any prime $p$ such that
 ${\mathcal S}_k(x)$ is irreducible in ${\mathbb F}_p[x]$. Furthermore, we show that ${\mathcal S}_k(x^p)$ is monogenic for any prime divisor $p$ of $k^2+3k+9$. These results extend previous work of the author on $k$-Wall-Sun-Sun primes.
   \end{abstract}

 %, power-compositional}

\maketitle
\section{Introduction}\label{Section:Intro}
For a monic polynomial $f(x)\in {\mathbb Z}[x]$ of degree $N$ that is irreducible over ${\mathbb Q}$, we say that $f(x)$ is \emph{monogenic} if  $\Theta=\{1,\theta,\theta^2,\ldots ,\theta^{N-1}\}$ is a basis for the ring of integers ${\mathbb Z}_K$ of $K={\mathbb Q}(\theta)$, where $f(\theta)=0$. Such a basis $\Theta$ is called a \emph{power basis}. If $\Theta$ is not a basis for ${\mathbb Z}_K$, we say that $f(x)$ is \emph{non-monogenic}. Since \cite{Cohen}
\begin{equation} \label{Eq:Dis-Dis}
\Delta(f)=\left[\Z_K:\Z[\theta]\right]^2\Delta(K),
\end{equation}
 where $\Delta(f)$ and $\Delta(K)$ denote, respectively, the discriminants over $\Q$ of $f(x)$ and $K$, we see that $f(x)$ is monogenic if and only if $\left[\Z_K:\Z[\theta]\right]=1$ or, equivalently, $\Delta(f)=\Delta(K)$.

 We also say that the number field $K$ is \emph{monogenic} if there exists some power basis for $\Z_K$. Consequently, while sufficient, the monogenicity\footnote{Although the terms \emph{monogenity} and \emph{monogeneity} are more common in the literature, we have chosen to use the more grammatically-correct term \emph{monogenicity}.} of $f(x)$ is not a necessary condition for the monogenicity of $K$. That is, there could exist $\alpha\in K$ and $g(x)\in \Z[x]$, with $g(x)\ne f(x)$, $\alpha\ne \theta$ and $g(\alpha)=0$, such that $g(x)$ is monogenic.

Let $k\ge 1$ be an integer, and let $(U_n)$
be the sequence defined by
\begin{align}\label{Eq:Shanks}
\begin{split}
U_0&=0,\quad U_1=0,\quad U_2=1 \quad\mbox{and}\\
 \quad U_n&=kU_{n-1}+(k+3)U_{n-2}+U_{n-3} \quad \mbox{for $n\ge 3$}.
 \end{split}
\end{align}
It is well known that $(U_n)$ is periodic modulo any integer $m\ge 2$, and we let $\pi(m)$
denote the length of this period. Note that $\pi(m)\ge 3$.
The following definition extends the idea of a $k$-Wall-Sun-Sun prime in \cite{JonesActa} to the sequence $(U_n)$ in \eqref{Eq:Shanks}.
\begin{defn}\label{Def:kS prime}
A prime $p$ is a \emph{$k$-Shanks prime} if $\pi(p^2)=\pi(p)$.
\end{defn}
\noindent Table \ref{T:1} gives some examples of $k$-Shanks primes $p$ and their period lengths $\pi(p)$.
\begin{table}[h]
 \begin{center}
\begin{tabular}{ccc}
 $k$ & $p$ & $\pi(p^2)=\pi(p)$\\ \hline \\[-8pt]
33 & 17 & 307\\ [2pt]
95 & 13 & 183\\ [2pt]
409 & 61 & 3783\\ [2pt]
618 & 43 & 1893\\ [2pt]
987 & 101 & 10303%\\ [2pt]
 \end{tabular}
\end{center}
\caption{$k$-Shanks primes $p$ and corresponding period lengths}
 \label{T:1}
\end{table}

 Note that the characteristic polynomial of $(U_n)$ is the polynomial
 \begin{equation}\label{Eq:Shankspoly}
 \S_k(x)=x^3-kx^2-(k+3)x-1,
 \end{equation}
 which we call a \emph{Shanks} polynomial, after Daniel Shanks who investigated $\S_k(x)$ \cite{Shanks}. Since $\S_k(1)=-2k-1\ne 0$ and $\S_k(-1)=-1\ne 0$, it follows from the Rational Root Theorem that $\S_k(x)$ is irreducible over $\Q$. Shanks called the cyclic cubic fields $\Q(\rho)$, where $\S_k(\rho)=0$, the ``simplest cubic fields". He observed many algebraic similarities between the cubic fields $\Q(\rho)$ when $k^2+3k+9$ is prime, and the quadratic fields $\Q(\alpha)$ when $k^2+4$ is prime, where $\alpha^2-k\alpha-1=0$. A straightforward calculation reveals that
 \begin{equation}\label{Eq:Delta(S)}
 \Delta(\S_k)=(k^2+3k+9)^2,
  \end{equation} and it is easy to see from \eqref{Eq:Dis-Dis}, as Shanks pointed out, that $\S_k(x)$ is monogenic when $k^2+3k+9$ is prime. We also see quite easily that $f(x)=x^2-kx-1$ is monogenic when $k^2+4$ is prime. It turns out that $\S_k(x)$ is monogenic in a more general setting (see Theorem \ref{Thm:Main}), as is $f(x)$ \cite{JonesActa}.

  We point out that Gras (in her PhD thesis) and Gras, Moser and Payan \cite{Gras3} also studied cyclic cubic fields around the same time as Shanks, and although there is considerable overlap in \cite{Shanks} and \cite{Gras3}, the methods used by Gras, Moser and Payan are different from Shanks.
  While a main focus in both \cite{Shanks} and \cite{Gras3} was the determination of the class numbers and class groups of the cyclic cubic fields, we are concerned here with the  monogenicity of the power-compositional polynomials $\S_k(x^p)$, where $p$ is a prime. In particular, we establish an intimate connection between the monogenicity of $\S_k(x^p)$ and whether $p$ is a $k$-Shanks prime, in the case when $\S_k(x)$ is irreducible in $\F_p[x]$. Furthermore, we show that $\S_k(x^p)$ is always monogenic when $p$ is a prime divisor of $k^2+3k+9$.
   To accomplish these tasks, we exploit some of the algebraic similarities between $\Q(\rho)$ and $\Q(\alpha)$, as noted by Shanks, to extend many of the ideas found in \cite{JonesActa} to the Shanks polynomials. More precisely, we prove
  \begin{thm}\label{Thm:Main}
 Let $k\ge 1$ be an integer such that $k\not \equiv 3 \pmod{9}$ and $\D$ is squarefree, where
  \begin{equation}\label{Eq:DD}
  {\mathcal D}:=\left\{\begin{array}{cl}
  (k/3)^2+k/3+1 & \mbox{if $k\equiv 0 \pmod{3}$,}\\
  k^2+3k+9 & \mbox{otherwise.}
  %\D:=\left\{\begin{array}{cl}
  %k^2+3k+9 & \mbox{if $k\not \equiv 0 \pmod{3}$}\\
  % (k/3)^2+k/3+1 & \mbox{if $k\equiv 0 \pmod{3}$.}
\end{array} \right.
\end{equation}
 Let $\S_k(x)$ be as defined in \eqref{Eq:Shankspoly} and let $p$ be a prime. 
  \begin{enumerate}
 \item \label{Main:I1} The polynomial $\S_k(x)$ is monogenic.
  \item \label{Main:I2} If $\S_k(x)$ is irreducible in $\F_p[x]$, then $p$ is a $k$-Shanks prime if and only if $\S_k(x^p)$ is non-monogenic.
  \item \label{Main:I3} If $k^2+3k+9\equiv 0 \pmod{p}$,
 then $\S_k(x^p)$ is monogenic.
 \end{enumerate}
 \end{thm}
\begin{rem}
  Although the monogenicity of $\S_k(x)$ follows from \cite{KS}, %[Theorem 1.1]
   we nevertheless provide a proof
  for the sake of completeness. See also \cite{Gras1,Gras2}.
\end{rem}
\section{Preliminaries}\label{Section:Prelim}
Throughout this article, we assume that
\begin{itemize}
\item $k\ge 1$ is an integer such that $k\not \equiv 3 \pmod{9}$,
\item $\D$, as defined in \eqref{Eq:DD}, is squarefree,
  \item $p$ and $q$ denote primes,
  \item $\S_k(x)$ is as defined in \eqref{Eq:Shankspoly},
  %\item $\S_p(x)$ is as defined in \eqref{Eq:SPC},
   \item $\ord_m(*)$ denote the order of $*$ modulo the integer $m\ge 2$.
   \end{itemize}

The following theorem, known as \emph{Dedekind's Index Criterion}, or simply \emph{Dedekind's Criterion} if the context is clear, is a standard tool used in determining the monogenicity of a monic polynomial that is irreducible over $\Q$.
\begin{thm}[Dedekind \cite{Cohen}]\label{Thm:Dedekind}
Let $K=\Q(\theta)$ be a number field, $T(x)\in \Z[x]$ the monic minimal polynomial of $\theta$, and $\Z_K$ the ring of integers of $K$. Let $q$ be a prime number and let $\overline{ * }$ denote reduction of $*$ modulo $q$ (in $\Z$, $\Z[x]$ or $\Z[\theta]$). Let
\[\overline{T}(x)=\prod_{i}\overline{t_i}(x)^{e_i}\]
be the factorization of $T(x)$ modulo $q$ in $\F_q[x]$, and set
\[g(x)=\prod_{i}t_i(x),\]
where the $t_i(x)\in \Z[x]$ are arbitrary monic lifts of the $\overline{t_i}(x)$. Let $h(x)\in \Z[x]$ be a monic lift of $\overline{T}(x)/\overline{g}(x)$ and set
\[F(x)=\dfrac{g(x)h(x)-T(x)}{q}\in \Z[x].\]
Then
\[\left[\Z_K:\Z[\theta]\right]\not \equiv 0 \pmod{q} \Longleftrightarrow \gcd\left(\overline{F},\overline{g},\overline{h}\right)=1 \mbox{ in } \F_q[x].\]
\end{thm}

The next two theorems are due to Capelli \cite{S}.
 \begin{thm}\label{Thm:Capelli1}  Let $f(x)$ and $h(x)$ be polynomials in $\Q[x]$ with $f(x)$ irreducible. Suppose that $f(\alpha)=0$. Then $f(h(x))$ is reducible over $\Q$ if and only if $h(x)-\alpha$ is reducible over $\Q(\alpha)$.
 \end{thm}

\begin{thm}\label{Thm:Capelli2}  Let $c\in \Z$ with $c\geq 2$, and let $\alpha\in\C$ be algebraic.  Then $x^c-\alpha$ is reducible over $\Q(\alpha)$ if and only if either there is a prime $p$ dividing $c$ such that $\alpha=\gamma^p$ for some $\gamma\in\Q(\alpha)$ or $4\mid c$ and $\alpha=-4\gamma^4$ for some $\gamma\in\Q(\alpha)$.
\end{thm}

The discriminant of $\S_k(x^p)$ given in the next proposition follows from the formula for the discriminant of an arbitrary composition of two polynomials which is due to John Cullinan \cite{Cullinan}. %, to the best of our knowledge, For a proof, see \cite{HJmonocyclo}.
\begin{prop}\label{Prop:discrim}
$\Delta(\S_k(x^p))=(-1)^{(p-2)(p-1)/2}p^{3p}(k^2+3k+9)^{2p}$.
\end{prop}

The next proposition follows from \cite{Shanks}.
\begin{prop}\label{Prop:Shanks1}
Suppose that $\S_k(\rho)=0$. Then the other zeros of $\S_k(x)$ are 
\[\sigma=\frac{-1}{\rho+1}=\rho^2-(k+1)\rho-2 \quad \mbox{and} \quad \tau=\frac{-(\rho+1)}{\rho}=-\rho^2+k\rho+k+2.\] % =\frac{-1}{\sigma+1}
The group of units of the cyclic cubic field $\Q(\rho)$ has rank 2, with  independent fundamental units $\rho$ and $\rho+1$.
\end{prop}
We have the following immediate corollary of Proposition \ref{Prop:Shanks1}.
\begin{cor}\label{Cor:Shanks1}
  The polynomial $\S_k(x)$ is either irreducible modulo $p$ or factors completely into linear factors modulo $p$.
\end{cor}

The next theorem follows from Corollary (2.10) in \cite{Neukirch}.
\begin{thm}\label{Thm:CD}%{\rm \cite{Neukirch}}
  Let $K$ and $L$ be number fields with $K\subset L$. Then \[\Delta(K)^{[L:K]} \bigm|\Delta(L).\]
\end{thm}

\section{The Proof of Theorem \ref{Thm:Main}}\label{Section:Main}
We first prove some lemmas.
 \begin{lemma}\label{Lem:Main0}
   The polynomial $\S_k(x^p)$ is irreducible over $\Q$.
  \end{lemma}
  \begin{proof}
       Assume, by way of contradiction, that $\S_k(x^p)$ is reducible. Suppose that $\S_k(\rho)=0$. Then, since $\S_k(x)$ is irreducible over $\Q$, we have that $\rho=\gamma^p$ for some $\gamma\in \Q(\rho)$, by Theorems \ref{Thm:Capelli1} and \ref{Thm:Capelli2}.
Then, we see by taking norms that
  \[\NN(\gamma)^p=\NN(\gamma^p)=\NN(\rho)=1,\] which implies that $\NN(\gamma)=1$, since $\NN(\gamma)\in \Z$ and $p\ge 3$. Thus, $\gamma$ is a unit and
  \[\gamma=\pm \rho^r(\rho+1)^s,\] for some $r,s\in \Z$, %since $\rho$ and $\rho+1$ are independent fundamental units
  by Proposition \ref{Prop:Shanks1}.
   Consequently,
  \[\rho=\gamma^p=(\pm 1)^p\rho^{rp}(\rho+1)^{sp},\] which implies that
  \[(\pm 1)^p\rho^{rp-1}(\rho+1)^{sp}=1,\] contradicting the independence of $\rho$ and $\rho+1$. %Hence, $\S_k(x^p)$ is irreducible over $\Q$.
  \end{proof}

\begin{lemma}\label{Lem:Main1}
 Let $\S_k(x)$ be irreducible in $\F_p[x]$. Suppose that $\rho$, $\sigma$ and $\tau$, as given in Proposition \ref{Prop:Shanks1}, are the zeros of $\S_k(x)$ in $\F_{p^3}$.   Then
 \begin{enumerate}
    \item \label{Per:I1} $\ord_m(\rho)=\ord_m(\sigma)=\ord_m(\tau)=\pi(m)$, where $m\in \{p,p^2\}$,
    \item \label{Per:I2} $p^2+p+1\equiv 0\pmod{\pi(p)}$,
    \item \label{Per:I3} $\pi(p^2)\in \{\pi(p),p\pi(p)\}$.
  \end{enumerate}
\end{lemma}
\begin{proof}
  For items \ref{Per:I1} and \ref{Per:I2}, it follows from \cite{Robinson} that the order, modulo an integer $m\ge 3$, of the companion matrix
  \[\CC=\left[\begin{array}{ccc}
    0&0&1\\
    1&0&k\\
    0&1&k+3
  \end{array}\right]\] for the characteristic polynomial $\S_k(x)$ of $(U_n)$ is $\pi(m)$. Since the eigenvalues of $\CC$ are $\rho$, $\sigma$ and $\tau$, we conclude that
  \[\ord_m\left(\left[\begin{array}{ccc}
    \rho&0&0\\
    0&\sigma &0\\
    0&0&\tau
  \end{array}\right]\right)=\ord_m(\CC)=\pi(m), \quad \mbox{for $m\in \{p,p^2\}$.}
  \] The Frobenius automorphism of $\F_p(\rho)$ transitively permutes the zeros of $\S_k(x)$, so that
  \[\{\rho,\sigma,\tau\}=\{\rho,\rho^p,\rho^{p^2}\}.\] Hence, $\ord_m(\rho)=\ord_m(\sigma)=\ord_m(\tau)$, which proves \ref{Per:I1}. Additionally,
  \[\rho^{p^2+p+1}\equiv \rho\rho^p\rho^{p^2} \equiv \rho\sigma\tau\equiv 1 \pmod{p},\]
  which establishes \ref{Per:I2}. 

  Item \ref{Per:I3} follows by applying, to the matrix $\CC$, an argument given in the proof of \cite[Proposition 1.]{Renault}.
  \end{proof}

\begin{lemma}\label{Lem:Main2}
Suppose that $\S_k(x)$ is irreducible in $\F_p[x]$ and that $\S_k(\rho)=0$, 
where $\rho\in \F_{p^3}$.
Let $\Z_L$ be the ring of integers of $L=\Q(\theta)$, where $\S_k(\theta^p)=0$.  Then
\[[\Z_{L}:\Z[\theta]]\equiv 0\pmod{p}\ \Longleftrightarrow \  \S_k(\rho^p)\equiv 0 \pmod{p^2}.\]
 \end{lemma}
\begin{proof}
We apply Theorem \ref{Thm:Dedekind} to $T(x):=\S_k(x^p)$ using the prime $p$.
 Since $\S_k(x)$ is irreducible in $\F_p[x]$, we have that $\overline{\S_k}(x^p)=(\S_k(x))^p$, and
 we can let
\[g(x)=\S_k(x) \quad \mbox{and}\quad h(x)=\S_k(x)^{p-1}.\]
 Then
     \[pF(x)=g(x)h(x)-T(x)=\S_k(x)^{p}-\S_k(x^p),\] so that
     \begin{equation}\label{Eq:pF}
     pF(\rho)\equiv \S_k(\rho)^p-\S_k(\rho^p)\equiv  -\S_k(\rho^p) \pmod{p^2}.
     \end{equation}
     Since $\S_k(x)$ is irreducible in $\F_p[x]$, we conclude that $\gcd(\overline{\S_k},\overline{F})\in \{1,\overline{\S_k}\}$ in $\F_p[x]$. Hence, it follows from \eqref{Eq:pF} that
     \begin{align*}
      [\Z_{L}:\Z[\theta]]\equiv 0\pmod{p}&\ \Longleftrightarrow \ \gcd(\overline{\S_k},\overline{F})=\overline{\S_k}\\
       &\ \Longleftrightarrow \ F(\rho)\equiv 0 \pmod{p}\\
       &\ \Longleftrightarrow \ \S_k(\rho^p)\equiv 0 \pmod{p^2},
     \end{align*}
          which completes the proof.
    \end{proof}

\begin{lemma}\label{Lem:Main3}
  If $\S_k(x)$ is irreducible in $\F_p[x]$, then
  \[\mbox{$p$ is a $k$-Shanks prime}\ \Longleftrightarrow \ \S_k(\rho^p) \equiv 0\pmod{p^2}.\]
   \end{lemma}
\begin{proof} Note that the zeros of $\S_k(x)$ in $(\Z/p\Z)[\rho]$ and $(\Z/p^2\Z)[\rho]$ are precisely $\rho$, $\sigma$ and $\tau$, as described in Proposition \ref{Prop:Shanks1}.

  Assume first that $p$ is a $k$-Shanks prime.
      Then $\pi(p^2)=\pi(p)$ from Definition \ref{Def:kS prime}, and we conclude from the proof of Lemma \ref{Lem:Main1} that $\rho^p \pmod{p^2}\in \{\sigma,\tau\}$. Without loss of generality, suppose that $\rho^p\equiv \sigma \pmod{p^2}$.
        Then 
       \[\S_k(\rho^p)=\rho^{3p}-k\rho^{2p}-(k+3)\rho^p-1\equiv \S_k(\sigma)\equiv 0 \pmod{p^2}.\]

       Conversely, assume that $\S_k(\rho^p)\equiv 0 \pmod{p^2}$.
       Hence, $\rho^p\pmod{p^2}\in \{\rho,\sigma,\tau\}$. 
Then, by  Lemma \ref{Lem:Main1}, if $\rho^p\equiv \rho \pmod{p^2}$, we have
\[1\equiv \rho^{p^2+p+1}\equiv (\rho^p)^p\rho^p\rho\equiv \rho^3 \pmod{p},\]
 so that $\pi(p)=3$. Thus,
 \[U_3=k\equiv 0 \pmod{p} \quad \mbox{and} \quad U_4=k^2+k+3 \equiv 0 \pmod{p},\]
 which implies that $p=3$ and $\S_k(x)\equiv (x-1)^3 \pmod{3}$, contradicting the fact that $\S_k(x)$ is irreducible in $\F_p[x]$. Therefore, without loss of generality, suppose that
 \[\rho^p\equiv \sigma \pmod{p^2}\quad \mbox{and} \quad \rho^{p^2}\equiv \tau \pmod{p^2}.\] Then
 \[\rho^{p^2+p+1}\equiv \rho^{p^2}\rho^p\rho\equiv \tau\sigma\rho \equiv 1 \pmod{p^2},\] and hence, $p^2+p+1\equiv 0 \pmod{\pi(p^2)}$. By Lemma \ref{Lem:Main1}, $\pi(p^2)\in \{\pi(p),p\pi(p)\}$. However, since $p^2+p+1\not \equiv 0 \pmod{p}$, it follows that $\pi(p^2)=\pi(p)$ and $p$ is a $k$-Shanks prime by Definition \ref{Def:kS prime}.
\end{proof}
Combining Lemma \ref{Lem:Main2} and Lemma \ref{Lem:Main3} yields
\begin{lemma}\label{Lem:Main4}
Let $\Z_L$ be the ring of integers of $L=\Q(\theta)$, where $\S_k(\theta^p)=0$.
  If $\S_k(x)$ is irreducible in $\F_p[x]$, then
\[\mbox{$p$ is a $k$-Shanks prime}\ \Longleftrightarrow \ [\Z_{L}:\Z[\theta]]\equiv 0\pmod{p}.\]
\end{lemma}

We are now in a position to present the proof of Theorem \ref{Thm:Main}.
\begin{proof}[Proof of Theorem \ref{Thm:Main}]
%\begin{proof}
Recall that $\S_k(x)$ is irreducible over $\Q$, and that $\Delta(\S_k)=(k^2+3k+9)^2$. Let $\Z_K$ be the ring of integers of $K=\Q(\rho)$, where $\S_k(\rho)=0$.

To establish item \ref{Main:I1}, we use \eqref{Eq:Delta(S)} and Theorem \ref{Thm:Dedekind} with $T(x):=\S_k(x)$ to show that $[\Z_K:\Z[\rho]]\not \equiv 0 \pmod{q}$, where $q$ a prime divisor of $k^2+3k+9$. Note that $q\ne 2$.

Suppose first that $q=3$. Then $3\mid k$ and $\overline{T}(x)=(x-1)^3$, so that we may let $g(x)=x-1$ and $h(x)=(x-1)^2$. Then
\begin{align*}
  3F(x)&=g(x)h(x)-T(x)\\ 
  &=(x-1)^3-(x^3-kx-(k+3)x-1)\\
  &=(k-3)x^2+(k+6)x,
  \end{align*}
  and $3F(1)=2k+3$.
  Therefore, if $3F(1)\equiv 0 \pmod{9}$, then $k\equiv 3 \pmod{9}$, which is a contradiction. Hence, $\gcd(\overline{g},\overline{F})=1$ in $\F_3[x]$, and $[\Z_K:\Z[\rho]]\not \equiv 0 \pmod{3}$.

  Suppose next that $q>3$. Note that $q\nmid k(k+3)$. Let $\alpha_1:=k/3$.
   Then
   \[\S_k(\alpha_1)=-\frac{(2k+3)(k^2+3k+9)}{27}\equiv 0 \pmod{q},\]
and, from Proposition \ref{Prop:Shanks1}, we have that the other zeros of $\S_k(x)$ modulo $q$ are
\begin{equation}\label{Eq:a3}
\alpha_2=-\frac{1}{\alpha+1}=\frac{3}{k+3} \quad \mbox{and}\quad \alpha_3=-\frac{\alpha+1}{\alpha}=-\frac{k+3}{k}.
\end{equation}
Since
\begin{align*}
  \alpha_1-\alpha_2&=\frac{k^2+3k+9}{3(k+3)}\equiv 0 \pmod{q} \mbox{ and}\\
  \alpha_1-\alpha_3&=-\frac{k^2+3k+9}{3k}\equiv 0 \pmod{q},
\end{align*}
we conclude that
\begin{equation}\label{Eq:Tbar1}
\overline{T}(x)=(x-\alpha_i)^3 \quad \mbox{for any $i\in \{1,2,3\}$}. %=(x-k/3)^3.
\end{equation}
   Since $k^2+3k+9\equiv 0 \pmod{q}$, then $(k+3/2)^2\equiv  -3(3/2)^2\pmod{q}$,
   which implies that $-3$ is a square modulo $q$. Thus,  $q\equiv 1 \pmod{3}$ by quadratic reciprocity. Observe then that
  \[\frac{k-jq}{3}\in \Z \quad \mbox{if} \quad k\equiv j \pmod{3},\quad \mbox{where}\quad j\in \{-1,0,1\}.\] Hence, from \eqref{Eq:Tbar1}, we can let
  \begin{equation}\label{Eq:g and h}
  g(x)=x-\frac{k-jq}{3} \quad \mbox{and} \quad h(x)=\left(x-\frac{k-jq}{3}\right)^2,
  \end{equation}
  if $k\equiv j \pmod{3}$, where $j\in \{-1,0,1\}$.
  We claim that
  \[qF\left(\frac{k-jq}{3}\right)\not \equiv 0 \pmod{q^2}.\] Since the details for each case are similar, we provide details only for the case $j=1$. Then,
   \[qF(x)=g(x)h(x)-T(x)=\left(x-\frac{k-q}{3}\right)^3-(x^3-kx^2-(k+1)x-1),\]
  so that
  \[qF\left(\frac{k-q}{3}\right)=\frac{(k^2+3k+9)(2k-3q+3)}{27}+\frac{q^3}{27}.\] Therefore, since $k^2+3k+9$ is squarefree,
  if $qF\left(\frac{k-q}{3}\right)\equiv 0 \pmod{q^2}$, we must have that $2k+3\equiv 0 \pmod{q}$, which implies that $k\equiv -3/2 \pmod{q}$. However, since $k^2+3k+9\equiv 0 \pmod{q}$, it follows that
  \[(-3/2)^2+3(-3/2)+9=27/4 \equiv 0 \pmod{q},\]
  which contradicts the fact that $q>3$. Thus, $qF\left(\frac{k-q}{3}\right)\not \equiv 0 \pmod{q^2}$. Consequently, $\gcd(\overline{g},\overline{F})=1$ in $\F_q[x]$ and hence, $[\Z_K:\Z[\rho]]\not \equiv 0 \pmod{q}$, from which we conclude that $\S_k(x)$ is monogenic, completing the proof of item \ref{Main:I1}.

We turn now to items \ref{Main:I2} and \ref{Main:I3}. Recall that $\S_k(x^p)$ is irreducible over $\Q$,  and that   \[\Delta(\S_k(x^p))=(-1)^{(p-2)(p-1)/2}p^{3p}(k^2+2k+9)^{2p}\] from, respectively, Lemma \ref{Lem:Main0} and Proposition \ref{Prop:discrim}. Let $\Z_L$ be the ring of integers of $L=\Q(\theta)$, where $\S_k(\theta^p)=0$. Since $\S_k(x)$ is monogenic, we have from \eqref{Eq:Dis-Dis} that $\Delta(K)=(k^2+3k+9)^2$. Thus,
  \[\Delta(L)\equiv 0 \pmod{(k^2+3k+9)^{2p}},\]
  by Theorem \ref{Thm:CD}.  Hence, we have shown that the monogenicity of $\S_k(x^p)$ is completely determined by the prime $p$. More explicitly, we have that
  \begin{equation}\label{Eq:Equiv Condition}
  \S_k(x^p) \mbox{ is monogenic}\ \Longleftrightarrow \ [\Z_L:\Z[\theta]]\not \equiv 0 \pmod{p}.
    \end{equation} Therefore, item \ref{Main:I2} is an immediate consequence of \eqref{Eq:Equiv Condition} and Lemma \ref{Lem:Main4}.

   To establish item \ref{Main:I3}, we use  Theorem \ref{Thm:Dedekind} with \[T(x):=\S_k(x^p)\quad  \mbox{and} \quad q:=p.\]
   Note that $p\ne 2$ since $k^2+3k+9\equiv 1 \pmod{2}$. The case $p=3$ can be handled in a manner similar to the situation $q=3$ for the monogenicity of $\S_k(x)$, and so we omit the details.

Suppose then that $p>3$. Note that $p\nmid k$ since $k^2+3k+9\equiv 0 \pmod{p}$. From \eqref{Eq:a3} and \eqref{Eq:Tbar1}, we deduce that
\[\S_k(x^p)\equiv \left(x-\alpha\right)^{3p} \pmod{p},\]
where $\alpha=-(k+3)/k$.

We claim that $\S_k(\alpha^p) \not \equiv 0 \pmod{p^2}$. By way of contradiction, assume that
\begin{equation}\label{Eq:Spalpha=0}
\S_k(\alpha^p)=\alpha^{3p}-k\alpha^{2p}-(k+3)\alpha^p-1\equiv 0 \pmod{p^2}.
\end{equation} Since
\[\alpha^3-1=-\frac{(2k+3)(k^2+3k+9)}{k^3}\equiv 0 \pmod{p},\] it follows that
$\alpha^{3p}\equiv 1 \pmod{p^2}$. Thus, we deduce from  \eqref{Eq:Spalpha=0} and \eqref{Eq:a3} that
\begin{equation}\label{Eq:Important Congruence}
\alpha^p\equiv -\frac{k+3}{k}\equiv \alpha \pmod{p^2}.
\end{equation} Applying \eqref{Eq:Important Congruence} to \eqref{Eq:Spalpha=0}, we conclude that
\[\alpha^{3}-k\alpha^{2}-(k+3)\alpha-1\equiv -\frac{(2k+3)(k^2+3k+9)}{k^3}\equiv 0 \pmod{p^2},\]
which is impossible since $2k+3\not \equiv 0 \pmod{p}$ and $\D$ is squarefree. Hence, the claim is established. 
Thus, since $\S_k^{\prime}(\alpha^p)\equiv 0 \pmod{p}$, it follows from Hensel that $\S_k(x^p)$ has no zeros modulo $p^2$.

Therefore, in particular, since we can let
\[g(x)=x+\frac{k+3+jp}{k} \quad \mbox{and} \quad h(x)=\left(x+\frac{k+3+jp}{k}\right)^{3p-1},\]
where $j\in \Z$ is such that $jp\equiv -3 \pmod{k}$, it follows from the claim that
\[pF(\alpha)=g(\alpha)h(\alpha)-\S_k(\alpha^p)=\left(\frac{jp}{k}\right)^{3p}-\S_k(\alpha^p)\equiv -\S_k(\alpha^p) \not \equiv 0 \pmod{p^2}.\]
Consequently, $[\Z_L:\Z[\theta]]\not \equiv 0 \pmod{p}$, which implies that $\S_k(x^p)$ is monogenic by \eqref{Eq:Equiv Condition}, and completes the proof of the theorem.
\end{proof}

Using a computer algebra system, it is easy to verify that $k\not \equiv 3 \pmod{9}$, $\D$ is squarefree, $\S_k(x)$ is irreducible in $\F_p[x]$  and $\S_k(x^p)$ is non-monogenic for each of the examples in Table \ref{T:1}.

%\section*{Acknowledgments}

%\section*{Data Availability Statement}
%The author confirms that all relevant data are included in the article.

\end{document}